\documentclass{amsart}
\usepackage[pdfauthor={Yiannis N. Petridis, Nicole Raulf, Morten S. Risager},
pdftitle={Quantum Limits of Eisenstein Series and Scattering states},
            pdfkeywords={Quantum Limits, Eisenstein series, Scattering poles},
             pdfcreator={Pdflatex}]
{hyperref}
\usepackage{xcolor}
\usepackage[normalem]{ulem}
\usepackage{marginnote}

\usepackage{amsfonts,amssymb,amsmath,amsthm}
\usepackage{url}
\usepackage{enumerate}
\DeclareMathOperator*{\res}{res}
\newtheorem{thrm}{Theorem}[section]
\newtheorem{lem}[thrm]{Lemma}

\newtheorem{cor}[thrm]{Corollary}
\theoremstyle{definition}

\newtheorem{remark}[thrm]{Remark}
\numberwithin{equation}{section}

\def \e {{\epsilon}}

\def \g {{\gamma}}
\def \G {{\Gamma}}
\def \R {{\mathbb R}}
\def \H {{\mathbb H}}
\def \C {{\mathbb C}}
\def \Z {{\mathbb Z}}

\def \Q {{\mathbb Q}}
\def \GmodH {{\Gamma\backslash\H^2}}
\def \GinfG {{\Gamma_\infty\backslash\Gamma}}

\newcommand{\abs}[1]{\left\lvert #1 \right\rvert}

\title[Limits of Eisenstein Series]{Quantum Limits of Eisenstein Series and Scattering states}    
\author{Yiannis N. Petridis}
\address{Department of Mathematics, University College London, Gower Street, London WC1E 6BT, United Kingdom}
\email{i.petridis@ucl.ac.uk}

\author{Nicole Raulf}
\address{Laboratoire Paul Painlev\'e ,
U.F.R. de Math\'ematiques, Universit\'e Lille 1
Sciences et Technologies,
59 655 Villeneuve d'Ascq C\'edex, France 
}
\email{raulf@math.univ-lille1.fr}

\author{Morten S. Risager}
\address{Department of Mathematical
  Sciences, University of Copenhagen, Universitetsparken 5, 2100
  Copenhagen \O, Denmark}
\email{risager@math.ku.dk}

\thanks{The first author was supported by NSF grant DMS-0401318, the second author was supported by a grant from DAAD (German Academic Exchange Service), and the third author was supported by a grant from  The Danish
Natural Science Research Council}
\date{\today}

\begin{document}
\begin{abstract} We identify the quantum limits of scattering states  for the modular surface. This is obtained through the study of quantum measures of  non-holomorphic Eisenstein series away from the critical line. We provide a range of stability   for  the quantum unique ergodicity theorem of Luo and Sarnak. 
\end{abstract}

\maketitle

\section{Introduction}
An important problem of quantum chaos is to describe the limiting
behaviour of eigenfunctions. On a compact negatively curved Riemannian manifold $X$
Shnirelman, Colin de Verdi\`ere, and Zelditch \cite{snirelman, deverdiere, zelditch1} have proved that for a 
{\lq generic\rq} family of eigenfunctions $\{\phi_j\}$ of the Laplacian the
associated measures $d\mu_j(z)=\abs{\phi_j(z)}^2d\mu(z)$ converge
weakly to the standard volume element $d\mu(z)$ of $X$, which we write
as 
\begin{equation}\label{weak-convergence}d\mu_j(z)\to d\mu{(z)} \textrm{ as }j\to \infty.\end{equation}
Zelditch \cite{zelditch2}
extended the result to finite volume hyperbolic surfaces. Lindenstrauss  and
Soundararajan \cite{lindenstrauss, sound} have proved that for $X=\GmodH$ where $\G\subset\
\hbox{PSL}_2(\Z)$ is of a certain \emph{arithmetic} type
$(\ref{weak-convergence})$ holds if $\phi_j$ runs through the set of  Hecke--Maa{\ss}
cusp forms. Earlier Luo and Sarnak \cite{luosarnak} investigated the question of quantum
chaos for Eisenstein series $E(z,1/2+it)$, i.e. \emph{generalized} eigenfunctions on
$X=\hbox{PSL}_2(\Z)\backslash\H^2$.   Since this series  is not square
integrable a certain normalization is needed. The actual statement  in \cite{luosarnak} is the following:
Let $A$ and $B$ be compact Jordan measurable subsets of $X$. Then
\begin{equation} \label{luosarnak que}\lim_{t\rightarrow\infty}\frac{\int_A|E(z, 1/2+it)|^2\, d\mu(z)}
{\int_B|E(z, 1/2+it)|^2\, d\mu(z)}=\frac{\mu(A)}{\mu(B)}.
\end{equation}
In fact,  see \cite{luosarnak}, this   follows from the result
\begin{equation}\label{limit}\int_A|E(z, 1/2+it)|^2\, d\mu(z)\sim \frac{6}{\pi}\cdot \mu(A)\log t, \quad t\to\infty.\end{equation}
A general cofinite subgroup likely has 
few embedded eigenvalues, possibly finite, so that (\ref{weak-convergence}) may be irrelevant.  So far  the quantum unique ergodicity of Eisenstein series is unproven for a general cofinite subgroup. A good substitute for the embedded eigenvalues are the scattering poles (resonances).  A natural question is to study the quantum limits of these scattering states. We address this question for $\G=\hbox{PSL}_2(\Z)$. As these states are not in $L^2(\GmodH)$, some normalization is also needed. 
Consider a simple pole $\rho$ of the scattering matrix. By the explicit calculation of the scattering matrix, see (\ref{scattering}), such
a pole is equal to half a zero of the Riemann zeta function. The
Eisenstein series has a pole at this point also and the residue has
Fourier expansion
\begin{equation}\label{scatteringfourierexpansion}
  \res_{s=\rho}E(z,s)=(\res_{s=\rho}\phi(s))y^{1-\rho}+\sum_{m\neq
  0}c_{m}\sqrt{y}K_{\rho-1/2}(2\pi m y)e^{2\pi i m x}.
\end{equation}
These scattering states are formal eigenfunctions of the Laplace
operator.
  We choose to normalize them as
follows: Set
\begin{equation}
  u_\rho(z)=(\res_{s=\rho}\phi(s))^{-1} \res_{s=\rho}E(z,s),
\end{equation}
so that the
scattering functions have the simplest possible growth behaviour at infinity, namely
$y^{1-\rho}$. 

We let $\{\gamma_n\}$ be a sequence of zeroes of the Riemann zeta
function  with $1/2 \leq \Re(\gamma_n)$ which satisfies $\lim_n
\gamma_n=\gamma_\infty<1$. Automatically $\Re(\gamma_n)<1$ and the Riemann hypothesis is equivalent to $\gamma_n=\gamma_\infty=1/2$, but we shall not
assume it. The points $\rho_n=\gamma_n/2$ are poles of the scattering matrix
$\phi(s)$.

\begin{thrm}\label{scatteringtheorem} Let $A$ be a compact Jordan measurable subset of $X$. Then  \begin{equation*}
    \int_{A}\abs{u_{\rho_n}(z)}^2d\mu(z)\to\int_AE(z,2-\gamma_\infty)d\mu(z)
  \end{equation*}
as $n\to\infty$.
This means that the quantum limit  of the measures $\abs{u_{\rho_n}(z)}^2d\mu(z)$ is the invariant, absolutely continuous measure $E(z,2-\gamma_\infty)d\mu(z)$.
\end{thrm} 

\begin{remark}
We know that a positive proportion of the zeros $\rho$ of $\zeta (s)$ lie on the critical line and are simple. In fact, this proportion is at least 40.58\%, see \cite{conrey}. 
Under the Riemann hypothesis and the conjectured simplicity of the Riemann zeros, there is only one quantum limit, the one described in the theorem above: $E(z, 3/2)d\mu (z)$.
\end{remark}

Theorem \ref{scatteringtheorem} follows rather easily by studying the quantum limits of Eisenstein series off the critical line. We present two such  theorems. The first addresses the stability of 
(\ref{luosarnak que})  if,  instead of real spectral value $1/4+t^2$, we move in the complex plane. To be precise let 
\begin{equation*}d\mu_{s(t)}(z)=|E(z, s(t))|^2\, d\mu(z),\end{equation*}
where \begin{equation*}s(t)=\sigma_t+it,\end{equation*} $ \sigma_t> 1/2 $. We investigate what happens
in the limit as $t \to \infty$, assuming that
$\sigma_{t}\to\sigma_\infty\geq 1/2$. We find qualitative differences
depending on whether $\sigma_{\infty}=1/2$ or not. If
$\sigma_\infty=1/2$ the situation is very similar to that of \cite{luosarnak}.

\begin{thrm}\label{maintheorem1}
 Assume $\sigma_\infty= 1/2$ {and $(\sigma_t-1/2)\log
 t\to 0$}. Let $A$, $B$ be compact Jordan measurable
 subsets of $X$. Then 
 \begin{equation*}
   \frac{\mu_{s(t)}(A)}{\mu_{s(t)}(B)}\to\frac{\mu(A)}{\mu(B)}
 \end{equation*}
as $t\to\infty$.
In fact  we have 
\begin{equation}
  \label{eq:2}
  \mu_{s(t)}(A)\sim\mu(A){\frac{6}{\pi}\log t}.
\end{equation}
\end{thrm}
  Theorem \ref{maintheorem1}
implies that the quantum unique ergodicity of Eisenstein series holds
in a quite big region in the complex plane (physical plane): for
spectral value $\lambda$ tending to infinity, the results holds as
long as $\Im (\lambda)={o(\sqrt{\Re {(\lambda)}}/\log\Re{(\lambda)})}$ in the region $\Re
(\lambda)\ge 0$. To see this we write 
$\lambda=s(1-s)$ with $s=\sigma+it$ and $\sigma\ge 1/2$. Then $\Im (\lambda)={o(\sqrt{\Re {(\lambda)}}/\log\Re{(\lambda)})}$
implies $\Re (\lambda)=\sigma(1-\sigma)+t^2\to \infty.$ We easily deduce that $\sigma$ is bounded. Then $(1-2\sigma)t=o(\sqrt{\sigma (1-\sigma)+t^2}){/\log{(\sigma(1-\sigma)+t^2))}}$ gives
$(\sigma-1/2)\log t\to 0$. So we can apply Theorem \ref{maintheorem1}.

Surprisingly the situation is qualitatively different when $\sigma_\infty>1/2$. We prove the following theorem:
\begin{thrm}\label{maintheorem2} Assume $\sigma_\infty>
  1/2$. Let $A$ be a compact Jordan measurable subset of $X$.  Then 
\begin{equation*}\mu_{s(t)}(A)\to\int_AE(z,2\sigma_\infty)d\mu(z)
\end{equation*}
as $t\to \infty$.
\end{thrm}
This proves that, when  $\sigma_\infty> 1/2$, the measures
$d\mu_{s(t)}$ do not become equidistributed. In fact it suggests in
this case to
consider different measures
\begin{equation}
  \label{eq:3}
  \nu_{s(t)}(z)=\abs{\frac{E(z,s(t))}{\sqrt{E(z,2\sigma_\infty)}}}^2d\mu (z).
\end{equation}
We note that, since $2\sigma_\infty>1$, we have
$E(z,2\sigma_\infty)>0$. The downside of this definition is 
that the function  $E(z,s(t))/\sqrt{E(z,2\sigma_\infty)}$ is not   an eigenfunction of the Laplacian  in contrast to $E(z, s(t))$. The upside is
that the corresponding measures become equidistributed:
\begin{cor}\label{maintheorem3} Assume $\sigma_\infty>  1/2$. Let $A$ be a compact Jordan measurable subset of $X$. Then 
\begin{equation*}\nu_{s(t)}(A)\to\mu(A), \quad t\to \infty.
\end{equation*}
\end{cor}
The result of Theorem \ref{maintheorem2} looks similar to Theorem 1 in \cite{guillarmou}, where the authors consider the equidistribution of Eisenstein series for convex co-compact 
subgroups $\Gamma$ of $\hbox{Iso} ({\H}^{n+1})$ with Hausdorff dimension of the limit set $\delta_{\Gamma}$ satisfying $\delta_{\Gamma}<n/2$. In both theorems the Eisenstein series $E(z, 2\sigma_{\infty})$ is well-defined and $E(z, 2\sigma_{\infty})d\mu(z)$ is the  quantum limit.

Similar results to Theorem \ref{scatteringtheorem} and Theorem \ref{maintheorem2} for more general surfaces with cuspidal ends have recently been announced by Dyatlov \cite{dyatlov}.

\begin{remark} 

The crucial ingredients  in \cite{luosarnak} are
\begin{enumerate}[(i)]
\item \label{one} A subconvex estimate for the $L$-series of a Maa{\ss} cusp form on its critical line, e.g. $L(\phi_j, 1/2+it)\ll (1+ |t|)^{1/3+\epsilon}$, see \cite{meurman}.

\item \label{two}A subconvex estimate for  the Riemann zeta function on its critical line, e.g. $\zeta (1/2+it)\ll (1+|t|)^{1/6+\epsilon}$.

\item \label{three} Estimates for $\zeta (1+it)$ and $(\zeta'/\zeta) (1+it)$.
\end{enumerate}

For Theorem \ref{maintheorem1} above we use subconvex bounds
  on $L$-functions and $\zeta(s)$. When
  $\sigma_\infty>1/2$, i.e. in Theorem \ref{maintheorem2} only \emph{convexity} bounds are used. While we use estimates on $\zeta(1+it)$ and $1/\zeta (1+it)$ in both cases, the estimate for $(\zeta'/\zeta )(1+it)$ is required only for the theorem of Luo--Sarnak. Our results clarify the mechanism for quantum unique ergodicity of Eisenstein series.
\end{remark}

\begin{remark} Equation (\ref{luosarnak que}) was extended by Jakobson \cite{jakobson} to the unit tangent bundle of $X$.  Koyama \cite{koyama} extended the result to Eisenstein series for $\hbox{PSL}_2(\Z[i])$, and   Truelsen \cite{truelsen} to Eisenstein series  for $\hbox{PSL}_2(\mathcal{O}_K)$, with $\mathcal{O}_K$ the integers of  a totally real field $K$ of finite degree over $\mathbb{Q}$ with narrow class number one. 

In both cases bounds of the type (\ref{one}), (\ref{two}), and  (\ref{three}) above are used. In the case of $K=\Q  (i)$  the subconvex estimate analogous to (\ref{one}) was established by Petridis and Sarnak \cite{petr4}, and the general case was established by Michel and Venkatesh \cite{michelvenkatesh}. As a substitute of (\ref{two}) and (\ref{three}) one uses estimates for the Dedekind zeta function $\zeta_{K}$. 

The analogous question for holomorphic Hecke cusp form of weight $k$  has recently been resolved
by Holowinsky and Soundararajan \cite{holowinsky}. Let $f_k$ be a  sequence of $L^2$-normalized holomorphic Hecke  cusp forms for the group $\hbox{SL}_2(\Z)$ of weight $k$ and let  $\phi_k(z):=y^{k/2}f_k(z)$. Then the measures $|\phi_k(z)|^2d\mu (z)$ converge weakly to $d\mu(z)$, as previously conjectured by Rudnick and Sarnak.  We note that in this
case $\phi_k$ is  eigenfunction of the weight $k$
Laplacian with eigenvalue $k/2(1-k/2)$.

\end{remark}

\section{Proofs}
The non-holomorphic Eisenstein series $E(z,s)$, $(z,s)\in \H^2\times \C$ is  defined for $\Re(s)>1$ by 
\begin{equation}
  \label{eq:1}E(z,s)=\sum_{\gamma\in\GinfG }\Im (\g z)^s.  
\end{equation}
Here
$\G=\hbox{PSL}_2(\mathbb{Z})$ and  $\G_\infty$ is the cyclic subgroup
generated by $z\mapsto z+1$.
 The Eisenstein series $E(z,s)$ admits a Fourier expansion of the cusp
 $i\infty$, see e.g. \cite[(3.25)]{iwaniec}
\begin{align}\nonumber E(z, s)&=\sum_{n\in\Z}a_n(y,s)e^{2\pi i n x}\\
&\label{eisenstein-expansion} =y^s+\phi (s)y^{1-s}+\frac{2y^{1/2}}{\xi
  (2s)}\sum_{n\neq  0}\abs{n}^{s-1/2}\sigma_{1-2s}(\abs{n})K_{s-1/2}(2\pi
\abs{n}y)e^{2\pi i nx}.\end{align}
Here $\xi (s)=\pi^{-s/2}\Gamma (s/2)\zeta (s)$ is the completed
Riemann zeta function satisfying the functional equation $\xi(s)=\xi(1-s)$,
$\sigma_{c}(n)$ is the sum of the $c$'th powers of the divisors of
$n$, and  $K_s(y)$ is the $K$-Bessel function. The scattering matrix is 
\begin{equation}\label{scattering}\phi (s)=\frac{\xi (2-2s)}{\xi (2s)}.\end{equation}
We notice that the corresponding expression in \cite{luosarnak}
is missing a factor of 2 in the non-zero terms. This is irrelevant for
their purpose but becomes crucial for ours. 

The spectral decomposition of $L^2(\Gamma\backslash \H^2 )$ allows us to consider
separately Maa{\ss} cusp forms and incomplete Eisenstein series.
\subsection{Maa{\ss} cusp forms} Since there is a basis of the cuspidal eigenspaces consisting of Hecke--Maa{\ss} cusp forms, we restrict our attention to those. 
\begin{lem}\label{maass estimates}Let $\phi_j$ be a Hecke--Maa{\ss} cusp form. Then 
  \begin{equation}
    \int_\GmodH\phi_j\abs{E(z,\sigma_t+it)}^2d\mu(z)\to 0,
  \end{equation}
as $t\to\infty$.
\end{lem}
\begin{proof}
  
The Maa{\ss} cusp form $\phi_j$ has a Fourier expansion
\begin{equation*}\phi_j (z)=y^{1/2}\sum_{n\ne 0} \lambda (n)K_{it_j}(2\pi n y)e ( nx),\end{equation*}
with $\lambda (1)=1$. We assume  that it is even, since, if it is odd, $\langle \phi_j, \mu_{s(t)}\rangle =0$.
 Being a Hecke eigenform, $\phi_j$ has an $L$-series with Euler product
\begin{equation*}L(\phi_j, s)=\sum_{n=1}^{\infty}\frac{\lambda (n)}{n^s}=\prod_{p}(1-\lambda (p)p^{-s}+p^{-2s})^{-1}.\end{equation*}
We want to understand the behavior as $t\rightarrow\infty$ of 
\begin{equation}
J_j(t)=\int_{\Gamma\backslash\H^2}\phi_j (z)\abs{E(z, s(t))}^2 d\mu(z).
\end{equation} 
We calculate 
\begin{equation}
I_j(s)=\int_{\Gamma\backslash\H^2}\phi_j (z)E(z, s(t))E(z, s) d\mu(z),
\end{equation} 
and  set $s=\overline{s(t)}$ to recover $J_j(t)$.
For fixed $s$, $I_j(s)$ is a  holomorphic function of $w=s(t)$. In \cite{luosarnak} this function is identified for $w=1/2+it$, so we use the principle of analytic continuation to deduce that
\begin{equation*}
I_j(s)=\frac{R(s)}{\xi (2s(t))}\frac{\prod\Gamma \left(\frac{s\pm it_j \pm (s(t)-1/2)}{2}\right)}{2\pi^s\Gamma (s)} ,
\end{equation*}
with
\begin{equation*}R(s)=\frac{L(\phi_j, s-s(t)+1/2)L(\phi_j, s+s(t)-1/2)}{\zeta (2s)}.\end{equation*}
We plug $s=\overline{s(t)}$ to get
\begin{align*}
J_j(t)&=&2^{-1}\pi^{s(t)-\overline{s(t)}}L(\phi_j, 1/2-2it)L(\phi_j, 2\sigma_t-1/2)\frac{\prod\Gamma \left(\frac{\overline{s(t)}\pm it_j \pm (s(t)-1/2)}{2}\right)}{|\Gamma (s(t))\zeta (2s(t))|^2} .
\end{align*}
We apply Stirling's formula \cite[5.112]{iwaniec-kowalski} in the form
\begin{equation}\label{stirling}
  \abs{\Gamma(\sigma+it)}=\sqrt{2\pi}\abs{t}^{\sigma-1/2}e^{-\frac{\pi}{2}\abs{t}}(1+O(\abs{t}^{-1}))
\end{equation}
uniformly for $\abs{\sigma}\le M$. Using this we find that the quotient of Gamma factors is $\ll_j \abs{t}^{1/2-2\sigma_t}$.

If $\sigma_t$ is bounded away from 1/2,  the function $\abs{\zeta(2s(t))}^{-2}$ is bounded and the convexity estimate 
$L(\phi_j, 1/2+it)\ll t^{1/2}$ suffices to guarantee that $\lim J_j (t)=0.$

If $\sigma_t$ is not bounded away from 1/2 and we need non-trivial estimates for $\zeta(2s(t))^{-1}$ and $L(\phi_j, 1/2+it)$ to reach the same conclusion.  Such estimates are certainly available: the estimate
\begin{equation}\label{bound}
 \log^{-1} \abs{t} \ll \abs{\zeta(2s(t))}\ll \log \abs{t}
\end{equation}
is classical in the theory of the Riemann zeta function (see \cite[3.6.5 and 3.11.8]{titchmarsh}), and the subconvexity estimate 
\begin{equation}
  L(\phi_j, 1/2+it)=O_{j,\e}(\abs{t}^{1/3+\e})
\end{equation} was proved by Meurman \cite{meurman}. We note that \emph{any} subconvexity estimate $L(\phi_j, 1/2+it)=O(\abs{t}^{1/2-\e})$ suffices to show that $\lim J_j (t)=0.$
\end{proof}

\subsection{Incomplete Eisenstein series}
We now concentrate on the contribution of the incomplete Eisenstein series. 
Let $h(y)\in C^{\infty}(\R^+)$ be a function which decreases rapidly at
$0$ and $\infty$. This means that $h(y)=O_N(y^N)$ for $0<y\le 1$ and 
$h(y)=O(y^{-N})$ for $y\gg 1$ for all $N\in \mathbb N$. Its Mellin transform is
\begin{equation}
H(s)=\int_0^{\infty}h(y)y^{-s}\frac{dy}{y}
\end{equation}
and the Mellin inversion formula gives
\begin{equation}\label{mellininversion}
h(y)=\frac{1}{2\pi i}\int_{a-i\infty}^{a+i\infty}H(s)y^s\,ds
\end{equation}
for any $a \in\R$. The function $H(s)$ is entire and $H(\sigma +it)$ is in the
 Schwartz space in the $t$ variable for any $\sigma\in\R$.
We consider the incomplete Eisenstein series
\begin{equation}
F_h(z)=\sum_{\gamma\in\Gamma_{\infty}\backslash\Gamma}h(\Im (\gamma z))=\frac{1}{2\pi i}\int_{a-i\infty}^{a+i\infty}H(s)E(z, s)\,ds.
\end{equation}
\begin{lem}\label{incomplete eisenstein estimates}
  Let $F_h$ be an incomplete Eisenstein series as above. Then 
  \begin{align*}  
     \int_\GmodH F_h(z)  \abs{E(z,\sigma_t+it)}^2&d\mu(z)\sim\\
&   \begin{cases}\int_\GmodH F_h(z) E(z,2\sigma_\infty)d\mu(z),&\textrm{
         if }\sigma_\infty\neq 1/2,\\
\int_\GmodH F_h(z)
d\mu(z){6\pi^{-1}\log t},
&\textrm{
         if }{(\sigma_t-1/2)\log t\to 0,}\end{cases}
  \end{align*}
as $t\to\infty$.
\end{lem}
\begin{proof}
We choose $a$ such that $a>2\sigma_t$ for all $t$. The function $F_h(z)$ is smooth and rapidly decreasing in the cusp. Then  unfolding and using Parseval we get
\begin{align}
\nonumber\int_\GmodH F_h(z)d\mu_{s(t)}(z)=&\int_{\GmodH}F_h(z)|E(z,s(t))|^2 d\mu(z) \\ \nonumber=&\int_0^{\infty}\int_0^1 h(y)|E(z,s(t))|^2 \frac{dx\, dy}{y^2}\\=&
\label{afterparseval}\int_0^\infty h(y) \left( \sum_{n\in \Z}\abs{a_{n}(y,s(t))}^2\right)\frac{dy}{y^2}.
\end{align}
\subsection{Contribution of the constant term.} By  
(\ref{eisenstein-expansion}) we have
\begin{equation*}
 \abs{a_0(y,s(t))}^2= y^{2\sigma_t}+2\Re(\phi(s(t))y^{1-2it})+\abs{\phi(s(t))}^2y^{2-2\sigma_t}.
\end{equation*}
We analyze the three terms separately. The first term is 
\begin{equation*}
 \int_0^{\infty}h(y)y^{2\sigma_t-1}\frac{dy}{y}=H(1-2\sigma_t),
\end{equation*}
which converges to $H(1-2\sigma_{\infty})$ when $t\to\infty$. Next
\begin{align*}
 \phi(s(t))\int_0^{\infty}h(y)y^{-2it}\frac{dy}{y}
=\phi(s(t))H(2it).
\end{align*}
The function $H(2it)$ decays rapidly and  $\phi(s(t))$ is bounded, see \cite[(8.6)]{selberg}. By analyzing the same expression with $\overline{\phi(s(t))}y^{1-2it}$ instead of $\phi(s(t))y^{1+2it}$ we find that the term in (\ref{afterparseval}) involving $\Re(\phi(s(t))y^{1-2it} )$ tends to zero.  
 
The last expression coming from the constant term is
\begin{align*}
 \abs{\phi(s(t))}^2\int_0^{\infty}h(y)y^{1-2\sigma_t}\frac{dy}{y}=\abs{\phi(s(t))}^2H(2\sigma_t-1).
\end{align*}
Certainly $H(2\sigma_t-1)\to H(2\sigma_\infty-1)$, as $t\to\infty$ and $\abs{\phi(s(t))}$ is bounded.

 Using the explicit expression for 
$\phi (s)$ in (\ref{scattering}) we have  better control of  the behavior of $\phi(s(t))$  when $\sigma_\infty\neq 1/2$. We have 
\begin{equation*}
  \abs{\phi(s(t))}=\frac{\abs{\xi(2-2{s(t)})}}{\abs{\xi(2s(t))}}=\pi^{2\sigma_t-1}\abs{\frac{\zeta(2(1-\sigma_t)-2it)}{\zeta(2\sigma_t+2it)}}\abs{\frac{\G(1-\sigma_t-it)}{\G(\sigma_t+it)}}.
\end{equation*}
Using the convexity bound $\zeta (\sigma+it)=O(\abs{t}^{(1-\sigma)/2+\e})$  we
get $$\zeta(2(1-\sigma_t)+it)=O( \abs{t}^{\sigma_t-1/2+\e}).$$ By (\ref{bound})
$$\frac{1}{\zeta (2\sigma_t+2it)}=O( \log \abs{t}).$$ The quotient of
$\G$-factors is asymptotic to   $\abs{t}^{1-2\sigma_t}$ by (\ref{stirling}). We therefore conclude that, when $\sigma_\infty\neq 1/2$, we have 
\begin{equation}\label{betterbound}
  \abs{\phi(s(t))}\to 0 
\end{equation}
as $t\to\infty$.

To summarize we have proved that the contribution of the constant term
in (\ref{afterparseval}) converges to  $H(1-2\sigma_\infty)$ if
$\sigma_\infty\neq 1/2$ and is $O(1)$ if $\sigma_\infty=1/2$.

\subsection{Contribution of  the non-constant terms.}  By  (\ref{eisenstein-expansion}) and (\ref{mellininversion}) the contribution equals
\begin{align*}
 A(t)&= \int_0^\infty\frac{1}{2\pi i}\int_{\Re(s)=a}H(s)y^sds \sum_{n=1}^\infty\frac{8y}{\abs{\xi(2s(t))}^2}n^{2\sigma_t-1}\abs{\sigma_{1-2s(t)}(n)}^2\abs{K_{s(t)-1/2}(2\pi ny)}^2\frac{dy}{y^2}\\
&=\int_0^\infty\frac{1}{2\pi i}\int_{\Re(s)=a}\!\!\!\!\!\!\! H(s)\sum_{n=1}^\infty\frac{y^s}{(2\pi n)^s} \frac{8}{\abs{\xi(2s(t))}^2}n^{2\sigma_t-1}\abs{\sigma_{1-2s(t)}(n)}^2\abs{K_{s(t)-1/2}(y)}^2ds\frac{dy}{y}\\
&=\frac{1}{2\pi i}\int_{\Re(s)=a}\!\!\!\!\!\!\! H(s)\frac{1}{(2\pi)^s} \frac{8}{\abs{\xi(2s(t))}^2}\sum_{n=1}^\infty \frac{\abs{\sigma_{1-2s(t)}(n)}^2}{n^{s-(2\sigma_t-1)}}\int_0^\infty y^s\abs{K_{s(t)-1/2}(y)}^2\frac{dy}{y}ds.\end{align*}
We now use \cite[6.576 (4)]{gradshteyn} to calculate the integral
involving the $K$-Bessel functions,  and the Ramanujan identity
\begin{equation*}
  \sum_{n=1}^\infty\frac{\sigma_a(n)\sigma_b(n)}{n^s}=\frac{\zeta(s)\zeta(s-a)\zeta(s-b)\zeta(s-a-b)}{\zeta(2s-a-b)}
\end{equation*}
to see that
\begin{align*}A(t)&=\frac{1}{2\pi i}\int_{\Re(s)=a}\!\!\!\!\!\!\! H(s) \frac{1}{\abs{\xi(2s(t))}^2}\frac{\xi(s-2\sigma_t+1)\xi(s-2it)\xi(s+2it)\xi(s+2\sigma_t-1)}{\xi(2s)}ds\\
&=\frac{1}{\abs{\xi(2s(t))}^22\pi i}\int_{\Re(s)=a}B(s)ds,
\intertext{where $B(s)$ equal $H(s)$ times the $\xi$-factors. Since $\xi(s)$ has poles at $s=0,1$  the poles of $B(s)$ in the region $\Re(s)\ge 1/2$ are at $1\pm 2it$, $2\sigma_t$, $2-2\sigma_t$, $\pm (2\sigma_t-1)$ and $\pm 2it$. We now move the line of integration to $\Re(s)=1/2$. By considering the Stirling asymptotics for the $\G$-factors, convexity bounds for the zeta functions, Eq.~(\ref{bound}), and the rapid decay of $H(s)$ we see that $B(s)$ decays rapidly in vertical strips and this allows to move the line of integration. We find that}
A(t)&=\frac{1}{\abs{\xi(2s(t))}^2}\left(\res_{s=1\pm 2it
}B(s)+\res_{s=2\sigma_t }B(s)\right.\\&\left.\quad+
\delta_{t}\cdot \res_{s=2-2\sigma_t
}B(s)+(1-\delta_{t})\cdot \res_{s=2\sigma_t -1}B(s) +\frac{1}{2\pi i}\int_{\Re(s)=1/2}B(s)ds\right),
\end{align*}
with $\delta_{t}=1$ if ${\sigma_t<3/4}$ and $0$ otherwise. We analyze these five terms.

(i) Using Stirling, convexity bounds on the zeta functions. (\ref{bound}) and the rapid decay of $H(1\pm 2it)$ the term \begin{equation*}
 \frac{1}{\abs{\xi(2s(t))}^2} \res_{s=1\pm 2it }B(s)=H(1\pm 2it)\frac{\xi(1\pm 4it)\xi(1\pm 2it-2\sigma_t+1)\xi(1\pm 2it+2\sigma_t-1)}{\abs{\xi(2s(t))}^2\xi(2\pm 4it)}
\end{equation*}
 tends to zero as $t\to\infty$.

(ii)  We now consider the second term:
\begin{equation*}
  \frac{1}{\abs{\xi(2s(t))}^2}\res_{s=2\sigma_t }B(s)=H(2\sigma_t)\frac{\xi(4\sigma_t-1)}{\xi(4\sigma_t)},
\end{equation*}
which converges to  $H(2\sigma_\infty)\frac{\xi(4\sigma_\infty-1)}{\xi(4\sigma_\infty)}$ when $t\to\infty$ and $\sigma_\infty\neq 1/2$. When $\sigma_t\to 1/2$ it behaves asymptotically like
\begin{equation*}
  H(1)\frac{1}{4\xi(2)(\sigma_t-1/2)}.
\end{equation*}

(iii)  We then move on to the third term:
\begin{align*}
  \frac{1}{\abs{\xi(2s(t))}^2}\res_{s=2-2\sigma_t }B(s)&=H(2-2\sigma_t)\frac{\xi({3}-4\sigma_t)\xi(2-2\sigma_t-2it)\xi(2-2\sigma_t+2it)}{\abs{\xi(2s(t))}^2\xi(4-4\sigma_t)}\\&=H(2-2\sigma_t)\abs{\phi(s(t))}^2\frac{\xi({3}-4\sigma_t)}{\xi(4-4\sigma_t)}.
\end{align*}
If $\sigma_\infty\neq 1/2$ we can use (\ref{betterbound}) to conclude
that this tends to zero, and if  $\sigma_\infty=1/2$ it 
{behaves like} \begin{equation*}
 { H(1)\frac{-\abs{\phi(s(t))}^2}{4\xi(2)(\sigma_t-1/2)}+O(1).}
\end{equation*}

(iv)  The fourth term is 
\begin{align*}
  \frac{1}{\abs{\xi(2s(t))}^2}\res_{s=2\sigma_t-1 }B(s)&=H(2\sigma_t-1)\abs{\phi(s(t))}^2.
\end{align*}
When $\sigma_\infty\neq 1/2$ this tends to zero and when
$\sigma_\infty=1/2$ it is bounded, by the same arguments as for the
third term.

(v) 
We now deal with the last term, i.e.
\begin{align*}
  \frac{1}{\abs{\xi(2s(t))}^22\pi i}&\int_{\Re(s)=1/2}B(s)ds=\frac{1}{2\pi \abs{\xi(2\sigma_t+2it)}^2}\\&\times\int\limits_{-\infty}^\infty H(1/2+i\tau)\frac{\abs{\xi(1/2+2\sigma_t-1+i\tau)}^2\xi(1/2+i(\tau-2t))\xi(1/2+i(\tau+2t))}{\xi(1+2i\tau)}d\tau.
\end{align*}
We note that $H(1/2+i\tau)$ is of rapid decay.  We study first the exponential behaviour of the integral as a function of $t$. Stirling asymptotics for the integrand give:
$$(e^{-\pi |\tau |/4})^2e^{-\pi |\tau/2-t|/2}e^{-\pi|\tau/2+t|/2}(e^{\pi|\tau|/2   })
\le e^{-\pi t},$$ which cancels with the exponential growth of $1/\abs{\xi(2s(t))}^2$. Using (\ref{bound}), the rapid decay of $H(1/2+i\tau)$, and any polynomial bound in $\tau$ of $\zeta(2\sigma_t-1/2+i\tau)$, we are reduced to estimate in $t$ the integral
$$ 
\frac{\log|t|}{(t^{-1/2+\sigma_t})^2}\int_{-\infty}^{\infty}\tilde H(\tau)(1+|\tau+2t|)^{-1/4}(1+|\tau-2t|)^{-1/4}|\zeta (1/2+i(\tau-2t))\zeta (1/2+i(\tau+2t))|\, d\tau,
$$where $\tilde H$ is some function of rapid decay. We separate now the two cases $\sigma_{\infty}>1/2$ and $\sigma_\infty=1/2$. In the first case we use the convexity bound on the $\zeta$ function to estimate the expression as
$$\frac{\log|t|}{(t^{-1/2+\sigma_t})^2}\int_{-\infty}^{\infty}\tilde H(\tau)(1+|\tau+2t|)^{\e}(1+|\tau-2t|)^{\e}\, d\tau  =o(1),$$
as $\sigma_{\infty}>1/2$. For the second case we can use any subconvex bound $\zeta (1/2+it)=O(\abs{t}^{1/4-\delta})$, for instance 
 Weyl's bound \cite[Theorem 5.5]{titchmarsh}
\begin{equation*}
  \zeta(1/2+it)\ll\abs{t}^{1/6+\e}.
\end{equation*}
We are reduced to estimate in $t$ the integral
\begin{align*}
\frac{\log|t|}{(t^{-1/2+\sigma_t})^2}&\int_{-\infty}^{\infty}\tilde H(\tau)(1+|\tau+2t|)^{-1/4+1/4-\delta}(1+|\tau-2t|)^{-1/4+1/4-\delta} d\tau,
\\=t^{1-2\sigma_t+\e}&\int_{-\infty}^{\infty}\tilde H(\tau)(1+|\tau+2t|)^{-\delta}(1+|\tau-2t|)^{-\delta}\, d\tau =o(1).  \end{align*}
This concludes the
evaluation of the non-constant terms in (\ref{afterparseval}).

To summarize we have proved that if $\sigma_\infty\neq 1/2$ the function $\int_\GmodH F_h(z)d\mu_{s(t)}(z)$ converges to
\begin{equation*}
  H(1-2\sigma_\infty)+H(2\sigma_\infty)\frac{\xi(4\sigma_\infty-1)}{\xi(4\sigma_\infty)}
\end{equation*}
 as $t\to\infty$, and, if  $\sigma_\infty= 1/2$, 
\begin{equation*}
 \int_\GmodH F_h(z)d\mu_{s(t)}(z) =
 {H(1)\frac{1-\abs{\phi(s(t))}^2} {4\xi(2)(\sigma_t-1/2)} +O(1)},
\end{equation*}
 as $t\to\infty$.
This finishes the proof once we notice that 
\begin{align*}
  H(1-2\sigma_\infty)+ & H(2\sigma_\infty)\frac{\xi(4\sigma_\infty-1)}{\xi(4\sigma_\infty)}\\
  &=\int_0^\infty
  h(y)\left(y^{(2\sigma_\infty-1)+1}+\phi(2\sigma_\infty)y^{-2\sigma_\infty+1}\right)\frac{dy}{y^2}\\
&=\int_0^\infty h(y)\left(\int_0^1E(z,2\sigma_\infty)dx\right)\frac{dy}{y^2}\\
&=\int_{\GmodH} F_h(z)E(z,2\sigma_\infty)d\mu(z)
\end{align*}
and
\begin{equation*}
  H(1)=\int_{\GmodH}F_h(z)d\mu(z),
\end{equation*}
while $$\frac{1-\abs{\phi(s(t))}^2}
{4\xi(2)(\sigma_t-1/2)}\sim\frac{6}{\pi}\log t.$$
The last claim is seen as follows: By using the mean value
theorem twice on the function $\sigma\mapsto
\phi(\sigma+it)\phi(\sigma-it)$  we
find $$\frac{1-\abs{\phi(\sigma+it)}^2}{1/2-\sigma}=\left(1-(1/2-\sigma')\abs{\phi(\sigma''+it)}^2\frac{\phi'}{\phi}(\sigma''\pm
  it)\right)\frac{\phi'}{\phi}(\sigma'\pm it) $$
for some $1/2\leq \sigma''\leq \sigma'\leq \sigma$,
where $\frac{\phi'}{\phi}(\sigma\pm
    it)$ means $\frac{\phi'}{\phi}(\sigma+
    it)+\frac{\phi'}{\phi}(\sigma''-
    it).$ The claim now follows from $(\sigma_t-1/2)\log t\to 0$
    and the well-known fact that \begin{equation}\label{phi/phi}\frac{\phi'}{\phi}(\sigma\pm
    it)\sim -4\log t\end{equation} as $t\to\infty$ for $\sigma \geq
  1/2$ and that
    $\abs{\phi(\sigma+it)}$ bounded for $\Re(s)\geq 1/2$ and $t>1$. The estimate (\ref{phi/phi}) follows from the bounds 
on the zeta function \cite[Theorem 5.17]{titchmarsh}, combined 
with the Stirling asymptotics on the Gamma function. 

\end{proof}

It is straightforward to verify that Theorem \ref{maintheorem1}, and
Theorem \ref{maintheorem2} follow from
Lemma \ref{maass estimates} and Lemma \ref{incomplete eisenstein
  estimates} using an approximation argument like in the proof of \cite[Proposition 2.3]{luosarnak}.

\begin{proof}[Proof of Theorem \ref{scatteringtheorem}] 
We have
\begin{align*}
 \abs{u_{\rho_n}}^2d\mu(z) 
&=\abs{(\res_{s=\rho_n}\phi(s))^{-1} \res_{s=\rho_n}E(z,s)}^2d\mu(z)\\ 
&=\abs{(\res_{s=\rho_n}\phi(s))^{-1} \res_{s=\rho_n}\phi(s)E(z,1-s)}^2d\mu(z)\\
&=\abs{E(z,1-\rho_n)}^2d\mu(z).
\end{align*}
The result now follows from Theorem \ref{maintheorem2} with $\sigma_\infty=1-\gamma_\infty/2$. 
\end{proof}

\begin{proof}[Proof of Corollary \ref{maintheorem3}]
Let $f$ be a test function for the convergence in Corollary \ref{maintheorem3}. Then we use $$\frac{f(z)}{E(z, 2\sigma_{\infty})}$$ as test function for Theorem \ref{maintheorem2} to deduce that, as ${t\to\infty}$
$$\int_{\GmodH}\frac{f(z)}{E(z, 2\sigma_{\infty})}d\mu_{s(t)}\to\int_{\GmodH}\frac{f(z)}{E(z, 2\sigma_{\infty})}E(z, 2\sigma_{\infty})d\mu(z)=\int_{\GmodH}f(z)\, d\mu(z).$$
Finally one uses the approximation argument in \cite[Proposition 2.3]{luosarnak} to complete the proof.

\end{proof}

\end{document}